\documentclass[letterpaper, 10 pt, conference]{ieeeconf}
\IEEEoverridecommandlockouts                              
\usepackage{xcolor,soul,framed} 

\colorlet{shadecolor}{yellow}
\usepackage[pdftex]{graphicx}
\graphicspath{{../pdf/}{../jpeg/}}
\DeclareGraphicsExtensions{.pdf,.jpeg,.png}

\usepackage[cmex10]{amsmath}
\allowdisplaybreaks
\usepackage{array}
\usepackage{mdwmath}
\usepackage{mdwtab}
\usepackage{eqparbox}
\usepackage{url}
\usepackage{comment}
\usepackage{algorithm}
\usepackage{algpseudocode}
\usepackage{amssymb}
\usepackage{cite}
\newtheorem{theorem}{Theorem}
\newtheorem{remark}{Remark}

\newcommand{\figref}[1]{Fig.~\ref{#1}}

\title{\LARGE \bf
Jump Law of Co-State in Optimal Control for State-Dependent Switched Systems and Applications
}

\author{Mi~Zhou$^1$ ~~~
      Erik~I.~Verriest$^1$ ~~~
      Yue~Guan$^2$ ~~~
      Chaouki~Abdallah$^1$
  \thanks{$^1$ Mi Zhou, Erik I. Verriest, and Chaouki Abdallah are with the School of Electrical and Computer Engineering, Georgia Institute of Technology. Email: {\tt\small \{mzhou91, erik.verriest, ctabdallah\} @gatech.edu }.}
  \thanks{$^2$ Yue Guan is with the School of Aerospace Engineering, Georgia Institute of Technology. Email: {\tt\small yguan44@gatech.edu }.}%
  }

\begin{document}

\maketitle
\begin{abstract}
This paper presents the jump law of co-states in optimal control for state-dependent switched systems.
The number of switches and the switching modes are assumed to be known a priori.
A proposed jump law is rigorously derived by theoretical analysis and illustrated by simulation results.
An algorithm is then proposed to solve optimal control for state-dependent hybrid systems.
Through numerical simulations, we further show that the proposed approach is more efficient than existing methods in solving optimal control for state-dependent switched systems.
\end{abstract}

%
\IEEEpeerreviewmaketitle


\section{Introduction} \label{sec:intro}
The optimal control of hybrid systems has been widely studied over the past few decades.
In practice, many dynamical systems have hybrid characteristics, such as fermentation processes \cite{Liu2014OptimalCO}, aerospace systems \cite{multiphase}, robots \cite{MITcheetah,Magnusrobots}, as well as social sciences \cite{skating} and natural systems \cite{snelllaw}. 
Generally, hybrid control systems are dynamical systems that exhibit interactions between continuous and discrete dynamics \cite{Liberzon2013SwitchedS}.
Switched systems are a particular class of hybrid dynamical systems composed of a family of continuous or discrete time subsystems and a law governing the transition between these subsystems.
For example, the following is a switched system:
\begin{eqnarray}
\dot x & = & A_\alpha x+B_\alpha u, \nonumber ~\\
y & = & C_\alpha x.
\end{eqnarray}
where $x \in \mathbb{R}^N$ is the state, $u\in \mathbb{R}^m$ is the control input, $y \in \mathbb{R}^p$ is the output, and $\alpha \in \left \{ 1,2,...,n \right \}$ is the switching mode or index.
The studies of hybrid switched systems have focused on their stability analysis \cite{Daafouz}, the optimal control of their switching modes \cite{FengSurvey,Stellato}, multi-mode switching \cite{Verriest_M3D}, as well as on the special class of piece-wise affine systems~\cite{PiecewiseAffine} and algorithms for solving such systems \cite{Shaikh, Xuping, Xuping2002}, etc.

Switched systems may be classified into state-dependent and time-dependent switched systems.
In a state-dependent switched system, the continuous state space is partitioned into a finite number of regions by several switching interfaces.
In each region, the system has continuous dynamics.
When the system trajectory hits a switching interface, the system dynamic ``switches" and the system state either jumps (i.e., impulse effects) or continues to evolve in a continuous fashion (i.e., no state jumps) at the switching point even though the trajectory may no longer be differentiable at that point.
In this paper, we will focus on switched systems without state jumps.
The scenario when a state-dependent switched system experiences multiple switches  (i.e., Zeno behavior) at the interface is out of the scope of this paper.

In this paper, we consider state-dependent switched systems whose state remains continuous at the switching interface and only crosses the switching interface once.
Our aim is to minimize (or maximize) a continuous (or discontinuous) performance index with Lagrange-type form for a fixed initial position, a given terminal time, and a-priori given switching interface. 

The following is a compilation of previous work towards the limited scope problem.
Reference \cite{Witsenhausen} formulated a general Mayer-type optimization problem and studied its dynamics, optimal control, the notion of well-behaved solution, and necessary conditions for optimality.
The article also provided geometric intuition of the jump conditions of the co-state, namely, $a \lambda(\tau-)+b\lambda(\tau+)+c \; \mathrm{grad}( g\left[x(\tau)\right])=0$, where $\lambda$ is the co-state, $\mathrm{grad} (g\left[x(\tau)\right])$ is the gradient of the switching interface at the switching state, $a$, $b$, $c$ are constants, and $\tau$ is the switching time instant.
Reference\cite{Shaikh} formulated a similar problem and employed classical variational and needle variation techniques to prove a set of necessary conditions for such systems, which showed the optimality condition related to the co-state satisfying $\lambda(t_s-)^o=\lambda(t_s+)^o + p \Delta_x m|_{t=ts}$, where $t_s$ is the switching time instant, $m(x,t)=0$ is the switching interface, and $p$ is a constant.
The authors then proposed a gradient descent algorithm to solve the optimal control problem indirectly.
Both papers provided a good qualitative analysis of the co-state, resulting in the optimality condition for the co-state.
However, the full characterization of the jump law for the co-state 
(e.g. the exact expression at the switching interface)
and its applications have never been thoroughly examined to the best of our knowledge.
Therefore, in this work, we explore the jump law of the co-state based on the results in \cite{Shaikh} and show how we can use such information to efficiently solve optimal control for hybrid systems.
Our contributions include the following:
\begin{itemize}
    \item We present jump laws of the co-state for the optimal control of general state-dependent switched systems, in both the case of time-invariant switching interface and time-varying switching interface.  
    We also provide rigorous theoretical analysis to prove the correctness of the jump laws.
    \item We then propose a new algorithm for solving optimal control of hybrid systems to illustrate the applicability of our theory.
    \item Finally, we compare the efficiency of our approach with existing algorithms.
\end{itemize}

This paper is organized as follows: 
In Section \ref{sec:Time-invarying}, we formulate the problem and propose the jump law of the co-state when the switching interface is time-invariant.
We then extend our approach to the scenario of the time-varying switching interface in Section \ref{sec:Time-varying}.
We present simulation results in Section \ref{sec:simulation} to verify the correctness of the jump laws.
Next, we show how the jump laws of co-state can actually be used to solve the optimal control problem of state-dependent switching system more efficiently by comparing with existing algorithm numerically in Section \ref{sec:application}.
Finally, we conclude our article in Section \ref{sec:conclusion} and envision our future work in this direction.
\section{Problem Description} \label{sec:Time-invarying}
We consider an optimal control problem exhibiting discontinuities in the system dynamics and in its performance index. 
Specifically, let
\begin{align} \label{eq:systemdynamic}
\dot x = f(x,u,t)=\left\{\begin{matrix}
f_{1}(x,u,t), \; \mathrm{if} \; g(x) \leq 0 \\
f_{2}(x,u,t), \; \mathrm{if} \; g(x)\geq 0
\end{matrix}\right.,
\end{align}
where $x\in \mathbb{R}^{n}$, $u\in \mathbb{R}^m$, $g(x)=0$ is the switching interface, the initial condition is $x(t_0)=x_0$, $g(x_0)<0$, and the desired final condition is $x(t_f)=x_f$, $g(x_f) > 0 $ at a fixed final time $t_f$.
Our objective is to minimize the performance index:
\begin{equation}\label{eq:systemPerformance}
J = \phi(x(t_f)) +\int_{t_0}^{t_f} L(x,u,t) {\rm d} t,
\end{equation}
where 
\begin{equation}
L(x,u,t) = \left\{\begin{matrix}
L_{1}(x,u,t),      \; \mathrm{if}\; g(x) \leq 0\\ 
L_{2}(x,u,t),     \; \mathrm{if} \; g(x) \geq  0
\end{matrix}\right.,
\end{equation}
and $\phi(x(t_f))$ is the terminal cost.
Note that $J$ is a function of the switching time instant $\tau$.
Given a switching time instant $\tau$, this optimal control problem may be regarded as a two-phase fixed final time optimal control problem.
\begin{remark}
In this paper, the switching interface is assumed to be differentiable.
The gradient of a scalar function, $g(x)$, is defined as a column vector $m(x) = \frac{\partial g(x)}{\partial x}$.
\end{remark}
\begin{remark} 
We make the physical restriction that the 
state variables must  be continuous at the interface, i.e., $x(\tau+) = x(\tau-)=x(\tau)$.
There are no jumps in the system itself, and therefore no impulsive causes.
\end{remark}
\begin{remark}
The theory can be extended to systems with multiple switching interfaces.
\end{remark}

We define the Hamiltonian of the system as usual: $H(x,\lambda,u,t) = L(x,u,t) + \lambda^{\top} f(x,u,t)$.
For a time-invariant system, the Hamiltonian is known to be constant~\cite{optimalcontrolbook}.
In the following, we will prove that the Hamiltonian is continuous with respect to time at the switching interface for two-modes systems.
The derivation uses standard variational principles and follows \cite{Verriest_M3D}.
\begin{theorem} \label{theorem:Hcontinuous}
The Hamiltonian $H$ for the two-modes problem \eqref{eq:systemdynamic}-\eqref{eq:systemPerformance} is continuous as a function of time when the state crosses the interface.
\end{theorem}
\begin{proof}
Consider the following cost function:
\begin{align*}
&J(\tau) = \phi(x(t_f)) +\int_{t_0}^{t_f} L(x(t), u(t),t) {\rm d}t \\
&=\phi(x(t_f))+\nu g(x(\tau)) +\\& \int_{t_0}^{\tau} L_{1} (x_{1}(t), u_{1}(t),t) {\rm d} t +\int_{\tau}^{t_f} L_{2}(x_{2}(t),u_{2}(t),t) {\rm d}t \tag{A} \label{eqn:Def_J_2} \\
&= \phi(x(t_f))+\nu g(x(\tau))+ \\& \int_{t_0}^{\tau} L_{1} (x_{1}(t), u_{1}(t),t) + \\& \lambda_{1}^\top(f_{1}(x_{1}(t),u_{1}(t),t)-\dot x_{1}(t)) {\rm d} t +\\& \int_{\tau}^{t_f} L_{2}(x_{2}(t),u_{2}(t),t)+ \\& \lambda^{\top}_{2}((f_{2}(x_{2}(t),u_{2}(t),t)-\dot x_{2}(t)) {\rm d}t\\
&= \phi(x(t_f))+\nu g(x(\tau))+\\&\int_{0}^\tau H_{1}(x(t),\lambda_1(t),u_1(t),t)-\lambda_1^\top \dot x_1(t) {\rm d}t +\\
& \int_{\tau}^{t_f} H_{2}(x_2(t),\lambda_2(t),u_2(t),t)-\lambda_2^\top\dot x_2(t) {\rm d}t,
\end{align*}
where $\nu$ is the Lagrange multiplier.
In \eqref{eqn:Def_J_2}, $\tau$ is the as-yet-unknown time when the interface is crossed.
The effect of a small variation $\delta \tau$, at the cross-over time, is
\begin{align*}
&J(\tau+\delta \tau) =  \phi(x(t_f))+\nu g(x(\tau+\delta \tau))+ \\& \int_{0}^{\tau+\delta \tau} H_{1}(x_{1}(t)+\delta x_1,\lambda_{1}(t),u_{1}(t)+\delta u_1,t){\rm d}t +\\& \int_{\tau+\delta \tau}^{t_f} H_{2}(x_{2}(t)+\delta x_2, \lambda_{2}(t),u_{2}(t)+\delta u_2,t) {\rm d}t \\& -\int_{0}^{\tau+\delta \tau} \lambda_{1}^\top(\dot x_{1}+\dot \delta x_1) {\rm d}t -\int_{\tau+\delta \tau}^{t_f} \lambda_{2}^\top
(\dot x_{2}+\dot \delta x_2) {\rm d}t \\
& = \phi(x(t_f))+\nu g(x(\tau+\delta \tau))+ \\& \int_{0}^{\tau} H_{1}(x_{1}(t)+\delta x_1,\lambda_{1}(t),u_{1}(t)+\delta u,t)- \\& \lambda_{1}^\top(t) (\dot x_{1}(t)+\dot \delta x_1) {\rm d}t +\\
&\int_{\tau}^{\tau+\delta \tau} H_{1}(x_{1}(t)+\delta x_1,\lambda_{1}(t),u_{1}(t)+\delta u,t)- \\ & \lambda_{1}^\top(t) (\dot x_{1}(t)+ \dot \delta x_1) {\rm d}t +\\
& \int_{\tau}^{t_f} H_{2}(x_{2}(t)+\delta x_2, \lambda_{2}(t),u_{2}(t)+\delta u_2,t) - \\&\lambda_{2}^\top(t) (\dot x_{2}(t)+\dot \delta x_2) {\rm d}t -\\&
\int_{\tau}^{\tau+\delta \tau} H_{2}(x_{2}(t)+\delta x_2, \lambda_{2}(t),u_{2}(t)+\delta u_2,t) - \\& \lambda_{2}^\top(t) (\dot x_{2}(t)+\dot \delta x_2) {\rm d}t
\end{align*}
The variation in $J$ due to variations in the control $u(t)$ and the cross-over time $\tau$ is
\begin{equation} \label{eq:deltaJ}
{\small 
\begin{aligned}
\delta J &= \left [ \left (\frac{\partial \phi}{\partial x}-\lambda^\top_{2}\right)\delta x \right]_{t_f}  +\left(\lambda_{2}-\lambda_{1}+\nu \frac{\partial g}{\partial x}|_{\tau} \right)\delta x+
\\& (H_{1}(x_{1}(\tau),u_{1}{\tau},\lambda_{1}(\tau),\tau) - H_{2}(x_{2}(\tau),u_{2}(\tau),\lambda_{2}(\tau),\tau))\delta \tau \\
&+ \int_{0}^{\tau} \left \{ \left( \frac{\partial H_{1}}{\partial x} + \lambda_{1}^{\top}\right )\delta x +\frac{\partial H_{1}}{\partial u}\delta u \right \} {\rm d}t \\
& +
\int_{\tau}^{t_f} \left \{ \left( \frac{\partial H_{2}}{\partial x} + \lambda_{2}^{\top}\right )\delta x +\frac{\partial H_{2}}{\partial u}\delta u \right \} {\rm d}t.
\end{aligned}}
\end{equation}
Note that $\delta x_{1,2}(\tau+\delta \tau) = \delta x_{1,2}(\tau)+\dot x_{1,2} \delta \tau$.
Since the initial state is fixed, $\delta x(t_0) = 0$.
The remaining parts of \eqref{eq:deltaJ} hold for arbitrary $\lambda$.
Choosing $\lambda$ such that
\begin{equation}\label{eqn:graidentperp}
\begin{aligned}
&\dot \lambda_{1}(t) = - \left(\frac{\partial H_{1}}{\partial x}\right)^\top, \quad t_0< t<\tau \\
&\dot \lambda_{2}(t) = -\left(\frac{\partial H_{2}}{\partial x}\right)^\top, \quad \tau< t<t_f \\
&\lambda_{2}(\tau)- \lambda_{1}(\tau) + \nu\frac{\partial g}{\partial x}|_{\tau} =0,\\
&\lambda_{2}(t_f) =\left( \frac{\partial \phi}{\partial x}\right)^\top \mid_{x=x_f},
\end{aligned}
\end{equation}
avoids the need to compute the induced state perturbations and substituting the state equations simplifies \eqref{eq:deltaJ} to
\begin{equation}
\begin{aligned}
\delta J &= \\ &\left[ H_{1}(x_{1}(\tau),u_{1}{\tau},\lambda_{1}(\tau),\tau) - H_{2}(x_{2}(\tau),u_{2}(\tau),\lambda_{2}(\tau),\tau)\right] \\ &\delta \tau  +
\int_{0}^{\tau} \frac{\partial H_{1}}{\partial u} \delta u {\rm d}t +\int_{\tau}^{t_f} \frac{\partial H_{2}}{\partial u} \delta u {\rm d}t.
\end{aligned}
\end{equation}
The necessary conditions for stationary of $J$ follow from $\delta J = 0$ and since $\delta u$ may be chosen arbitrarily, by virtue of the fundamental lemma, the integral condition lifts to the condition $\frac{\partial H_{2}}{\partial u}= 0$, $\frac{\partial H_{1}}{\partial u} = 0$ and 
\begin{equation}
    H_{1}(x_{1}(\tau), \lambda_{1}(\tau), u_{1}(\tau),\tau) =  H_{2}(x_{2}(\tau), \lambda_{2}(\tau), u_{2}(\tau),\tau),
\end{equation}
expressing that the Hamiltonian $H$ should be continuous as a function of time when the state crosses the interface \footnote{This theorem remains valid for free terminal state optimal control problems.}.
\end{proof}

Based on Theorem \ref{theorem:Hcontinuous}, we obtain that at the cross-over time, $H_{2}^o(x_{2}(\tau),\lambda_{2}(\tau),u_{2}(\tau)) = H_{1}^o(x_{1}(\tau),\lambda_{1}(\tau),u_{1}(\tau))$, i.e., $(L_{2}(\tau) +\lambda_{2}^{\top}(\tau) f_{2}(\tau))^o = (L_{1}(\tau)+\lambda_{1}^{\top}(\tau) f_{1}(\tau))^o$.
To simplify our notation, we omit the superscript $o$ and the argument $(\tau)$, and the following equations are all with $o$ and $(\tau)$, which means under the optimal control and at the switching time~$\tau$.
Define $\langle f \rangle = \frac{1}{2}(f_{2}(\tau) + f_{1}(\tau))$, $\Delta f = f_{2}(\tau)-f_{1}(\tau)$, $\Delta \lambda = \lambda_{2}(\tau)-\lambda_{1}(\tau)$, and $\langle \lambda \rangle = \frac{1}{2}(\lambda_{2}(\tau)+\lambda_{1}(\tau))$, we have
\begin{equation} \label{eq:form1}
\begin{aligned}
& L_{2}+\lambda_{2}^{\top} f_{2} = L_{1}+\lambda_{1}^{\top} f_{1} \\
& \Rightarrow \Delta L  = -\Delta \lambda^{\top} f_{1}-\lambda_{2}^{\top} \Delta f.
\end{aligned}
\end{equation}
Similarly,
\begin{equation}\label{eqn:form2}
\begin{aligned}
& L_{2}+\lambda_{2}^{\top} f_{2} = L_{1}+\lambda_{1}^{\top} f_{1} \\
& \Rightarrow \Delta L = -\lambda_{1}^{\top} \Delta f-f_{2}\Delta \lambda^{\top}.
\end{aligned}
\end{equation}
Adding \eqref{eq:form1} to \eqref{eqn:form2} leads to
\begin{equation} \label{eq:lefteqright}
\Delta \lambda^{\top} \langle f \rangle = - \left(\Delta L + \langle \lambda^{\top} \rangle \Delta f \right).
\end{equation}
Multiply both sides of \eqref{eq:lefteqright} by $\left(\frac{\partial g(x)}{\partial x}|_{x^o} \right)$ (i.e., the gradient w.r.t. the state at the interface under the optimal control) to obtain
\begin{equation} \label{eq:approching}
\Delta \lambda^{\top} \langle f \rangle \left(\frac{\partial g(x)}{\partial x}|_{x^o} \right) = - \left(\Delta L + \langle \lambda^{\top} \rangle \Delta f \right)\left(\frac{\partial g(x)}{\partial x}|_{x^o} \right).
\end{equation}
The optimality condition derived in \eqref{eqn:graidentperp} shows that under the optimal control, $\Delta \lambda \in \left( T_{x^*} g(x)\right )^\perp$ (i.e., $\Delta \lambda$ is parallel to the gradient of the interface), thus we can change the left side of \eqref{eq:approching} to the following, 
\begin{equation} \label{eq:changeorder}
\Delta \lambda^{\top} \langle f \rangle \left(\frac{\partial g(x)}{\partial x}|_{x^o} \right) = \left(\frac{\partial g(x)}{\partial x} |_{x^o}\right)^\top \langle f\rangle \Delta \lambda.
\end{equation}
\begin{proof}
Denote $m = \left(\frac{\partial g(x)}{\partial x}|_{x^o} \right)$.
First, we show that
\begin{align*}
\Delta \lambda^\top \langle f\rangle m = m^\top \langle f\rangle \Delta \lambda.
\end{align*}
Since $m$ is parallel to $\Delta \lambda$, we can write them respectively as $m = t_1 v$, $\Delta \lambda = t_2 v$, where $v$ is a direction vector, $t_1$, $t_2$ are scalars.
Then $\Delta \lambda^\top \langle f\rangle m = t_1 v^\top\langle f\rangle t_2 v=t_1 t_2 v^\top\langle f\rangle v$, $ m^\top \langle f\rangle \Delta \lambda  =t_1 v^\top \langle f\rangle t_2 v = t_1 t_2 v^\top\langle f\rangle v$.
\end{proof}
Substitute \eqref{eq:changeorder} into \eqref{eq:approching}, we have the following equation:
\begin{equation}
\left(\frac{\partial g(x)}{\partial x}|_{x^o} \right)^\top \langle f\rangle \Delta \lambda = - \left(\Delta L + \langle \lambda^{\top}\rangle \Delta f \right)\left(\frac{\partial g(x)}{\partial x} |_{x^o}\right),  
\end{equation}
which gives the relation
\begin{equation} \label{eq:finalresult}
\Delta \lambda = -\frac{\left(\Delta L + \langle \lambda^{\top}\rangle \Delta f \right)m}{ m^\top \langle  f\rangle}.
\end{equation}
where $m = \frac{\partial g(x)}{\partial x}|_{x^o}$.
This relation specifies the quantitative behavior of the co-states at the interface.
\section{Time-varying switching interface} \label{sec:Time-varying}
In this section, we present the jump law of co-state $\Delta \lambda$ under the time-varying (TV) switching interface.
\subsection{Problem formulation}
Consider the following general switched system with a time-varying switching interface $g(x,t) = 0$:
\begin{align} \label{eqn:oriTVSystem}
\dot x = \left\{\begin{matrix}
f_{1}(x,u,t), \; \mathrm{if} \; g(x,t)\leq 0 \\ 
f_{2}(x,u,t),\; \mathrm{if} \; g(x,t) \geq 0
\end{matrix}\right. ,
\end{align}
and stage cost
\begin{align} \label{eqn:oriTVPerformance}
L = \left\{\begin{matrix}
L_{1}(x,u,t), \; \mathrm{if} \;  g(x,t) \leq 0 \\ 
L_{2}(x,u,t), \; \mathrm{if} \;g(x,t) \geq 0
\end{matrix}\right. .
\end{align}
Suppose the initial state $x_0$, the initial time $t_0$, and final time $t_f$ are given.
\subsection{Theoretical analysis}
Let the time-varying switching interface be given by $m^T x+ \mu t +c = 0$ \footnote{Note that, in general, the interface may be nonlinear, in which case, this surface is the tangent plane at the switching point.}.
We change this system into a time-invariant (TIV) system by conducting a variable augmentation.  Define the new augmented variable $\Bar{x}=[x,t]^\top$.
Then the original system \eqref{eqn:oriTVSystem} is turned into the following time-invariant system:
\begin{align}
\dot  {\Bar{ x}} =  \begin{bmatrix}
f(\bar{x},u)\\ 
1
\end{bmatrix} =    F(\bar{x},u).
\end{align}
The stage cost is $\bar{L}(x,u,t) = L(\bar{x},u)$.
The new Hamiltonian is then $\bar{H}(\bar{x},u,\bar{\lambda})=\bar{L}(\bar{x},u)+\bar{\lambda}^\top F(\bar{x},u)$.
In this case, the optimal control condition gives us
\begin{align*}
&\frac{\partial{\bar{H}}}{\partial u} = 0 \\ \Rightarrow\qquad &\frac{\partial{\bar{L}}}{\partial u}+\bar{\lambda}^\top\frac{\partial F}{\partial u} = 0 \\
 \Rightarrow \qquad &\frac{\partial L}{\partial u} +[\lambda^\top, \lambda_t]\begin{bmatrix}
\frac{\partial f}{\partial u}\\ 0
\end{bmatrix} =\frac{\partial L}{\partial u} + \lambda^\top\frac{\partial f}{\partial u} = 0.
\end{align*}
The Euler-Lagrangian equation gives us

\resizebox{.9\linewidth}{!}{
\begin{minipage}{\linewidth}
\begin{align} 
    \dot {\bar{\lambda}} = \begin{bmatrix}
    \dot \lambda\\ 
    \dot \lambda_t
    \end{bmatrix} = -\left(\frac{\partial \bar{H}}{\partial x}\right)^\top = \begin{bmatrix}
    -\left(\frac{\partial \bar{H}}{\partial \bar{x}}\right)^\top\\ 
    -\frac{\partial \bar{H}}{\partial t} 
    \end{bmatrix}=
    \begin{bmatrix}
    -\left(\frac{\partial H}{\partial x}\right)^\top\\ 
    -(\frac{\partial L}{\partial t}+\lambda^\top \frac{\partial f}{\partial t})
    \end{bmatrix}.
\end{align}
\end{minipage}
}

\noindent
That is
\begin{align*}
\dot \lambda &= -\left(\frac{\partial H}{\partial x}\right)^\top, \\
\dot \lambda_t &= -\left(\frac{\partial L}{\partial t} + \lambda^\top \frac{\partial f}{\partial t}\right).
\end{align*}
With this augmented scheme and using the jump law in~\eqref{eq:finalresult}, we obtain the following law:
\begin{align} \label{eqn:jumplawTV}
    \Delta \lambda = - \frac{(\Delta L+\langle \lambda \rangle \Delta f)m}{ \langle m^T f + \mu\rangle},
\end{align}
where the notation $\langle \cdot \rangle$ means the average; $\mu$ is the partial derivative of the time-varying switching interface $m^\top x+\mu t +c = 0$ with respect to time $t$.
\section{Illustrative simulation examples} \label{sec:simulation}
In this section, we solve this optimal control problem by doing a brute force search of the switching time and then verify our proposed co-state jump law \eqref{eq:finalresult} and \eqref{eqn:jumplawTV}.
\subsection{Second-order system with time-invariant switching interface}
In this scenario, we consider the following second-order system:
\begin{align*}
\mathrm{s1}:
\left\{\begin{matrix}
\dot x_1 = 2 x_2\\ 
\dot x_2 = u
\end{matrix}\right.  , \quad  x_1^2+  x_2^2 \leq 1,
\end{align*}
\begin{align}
\mathrm{s2}:
\left\{\begin{matrix}
\dot x_1 = x_2\\ 
\dot x_2 = u
\end{matrix}\right. , \quad x_1^2+  x_2^2\geq 1 . 
\end{align}
The stage cost is given by:
\begin{align*}
L= \left\{\begin{matrix}
\frac{1}{8}u^2  , \; x_1^2+ x_2^2 \leq 1\\ \\
 \frac{1}{2}u^2,\; x_1^2+ x_2^2\geq 1
\end{matrix}\right. .
\end{align*}
We consider the boundary conditions: $x_0 = [0, 0]^\top$, $x_f = [
2 , 2 ]^\top$ and $t_f=2$.
Here the switching interface is a quarter of a unit circle.
The switching time instant with optimal $J=0.7382$ is $\tau = 0.8881$.
The co-state jump is $\Delta \lambda =[-0.0840, -0.1807]^\top$.
The $\Delta \lambda$ calculated using \eqref{eq:finalresult} is $\Delta \lambda = [ -0.0840
    , -0.1807]^{\top}$.
The optimal trajectory $x$ and the co-state $\lambda$ are shown in~\figref{fig:2orderTIVsystem} and~\figref{fig:2orderTIVsystem_lambda} respectively.
\begin{figure} 
  \begin{center}
  \includegraphics[width=3.5in]{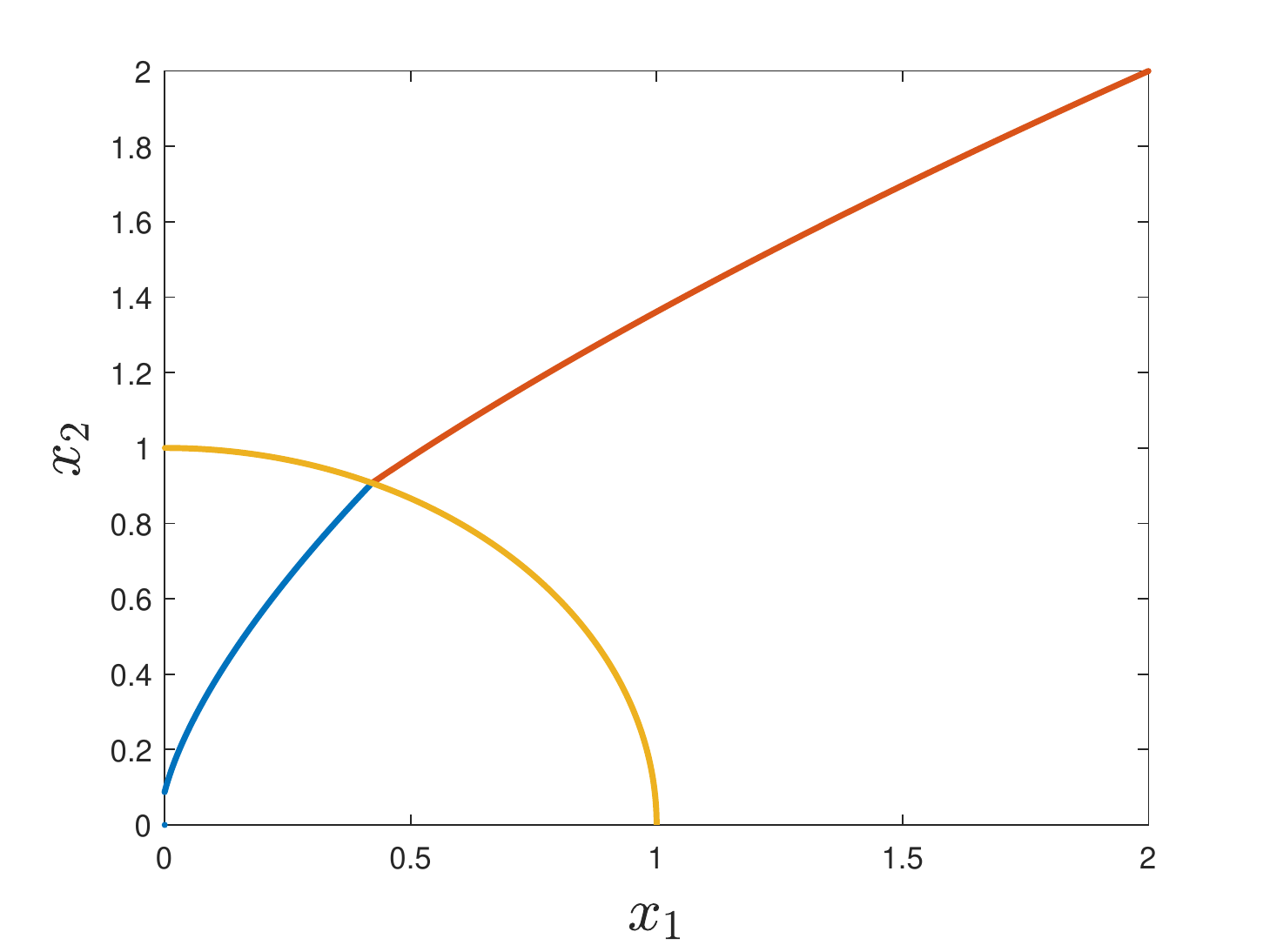}
  \caption{Second order system with time-invariant switching interface: $x$.}
  \label{fig:2orderTIVsystem}
  \end{center}
  \vspace{-15pt}
\end{figure}

\begin{figure} 
  \begin{center}
  \includegraphics[width=3.5in]{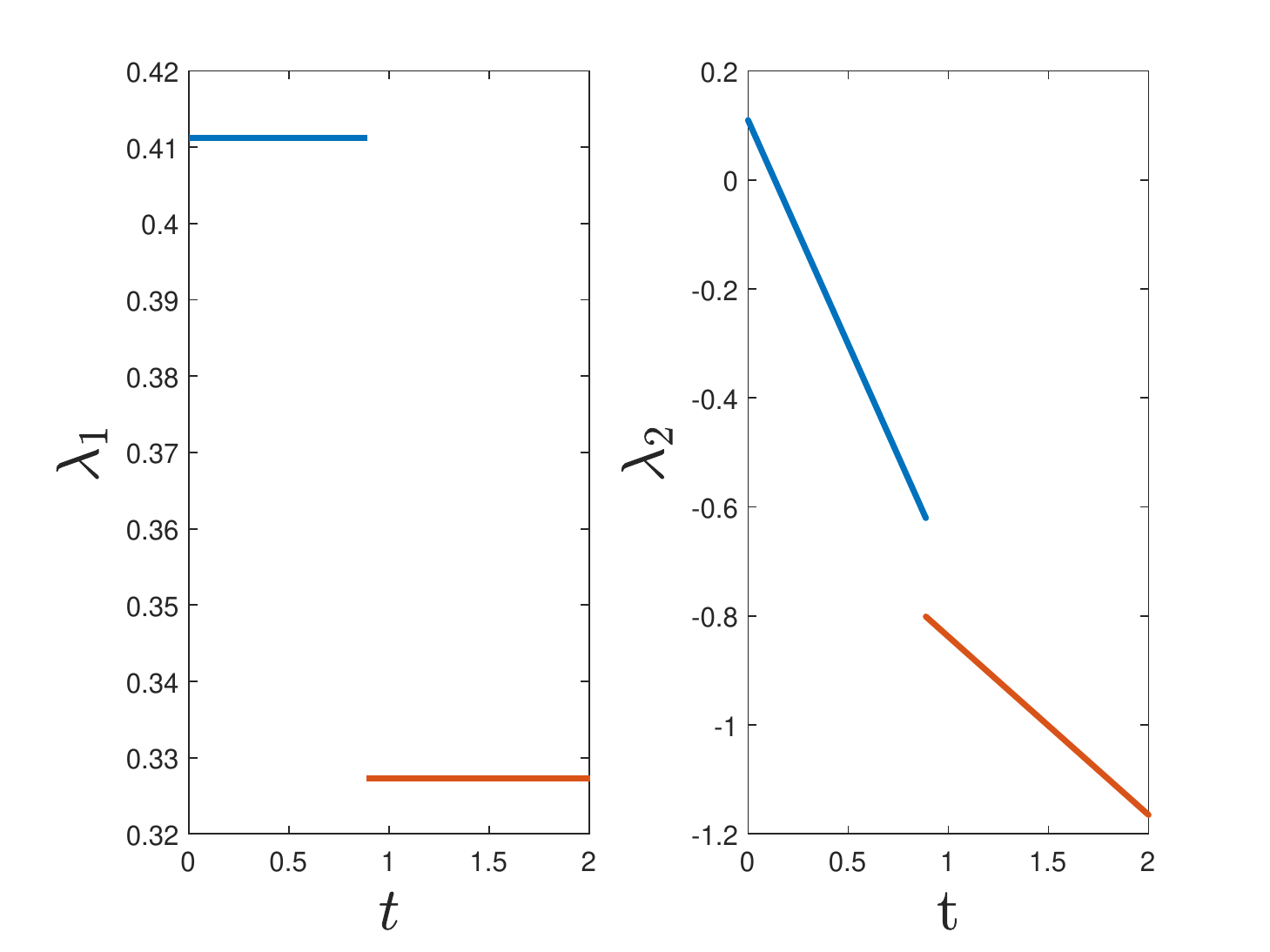}
  \caption{Second-order system with time-invariant switching interface: $\lambda$.}
  \label{fig:2orderTIVsystem_lambda}
  \end{center}
  \vspace{-20pt}
\end{figure}
\subsection{Second-order switching system with time-varying switching interface}
Consider the following second-order system with the time-varying switching interface $g(x,t)=x_1+x_2+t -1= 0$:
\begin{align}
\mathrm{s1}: \quad
\left\{\begin{matrix}
\dot x_1 = 2 x_2\\ 
\dot x_2 = u
\end{matrix}\right.    , \quad  x_1+x_2 + t\leq 1 
\end{align}
switches to 
\begin{align}
\mathrm{s2}: \quad
\left\{\begin{matrix}
\dot x_1 = x_2\\ 
\dot x_2 = u
\end{matrix}\right.  , \quad x_1+ x_2 + t \geq 1  
\end{align}
The discontinuous stage cost is 
\begin{align}
L= \left\{\begin{matrix}
u^2 ,    \quad x_1+x_2+t \leq 1\\ 
\\ \frac{1}{2}u^2,        \quad x_1+x_2+t\geq 1
\end{matrix}\right. .
\end{align}
We consider the boundary conditions: $x_0 = [
0,0]^\top$, $x_f = [
2 ,2 ]^\top$ and $t_f=2$.
The numerical solution gives the optimal switching time instant as $\tau=0.4790 $ with $J=1.1539$ and $\Delta \lambda= [0.1317,0.1318]^\top$.
$\Delta \lambda$ using jump law of \eqref{eqn:jumplawTV} is $[0.1318, 0.1318]^\top$.
The optimal trajectory $x$ and the co-state $\lambda$ are shown in~\figref{fig:tvswitching} and~\figref{fig:tvswitching_lambda} respectively.
\begin{figure}
  \begin{center}
  \includegraphics[width=3.5in]{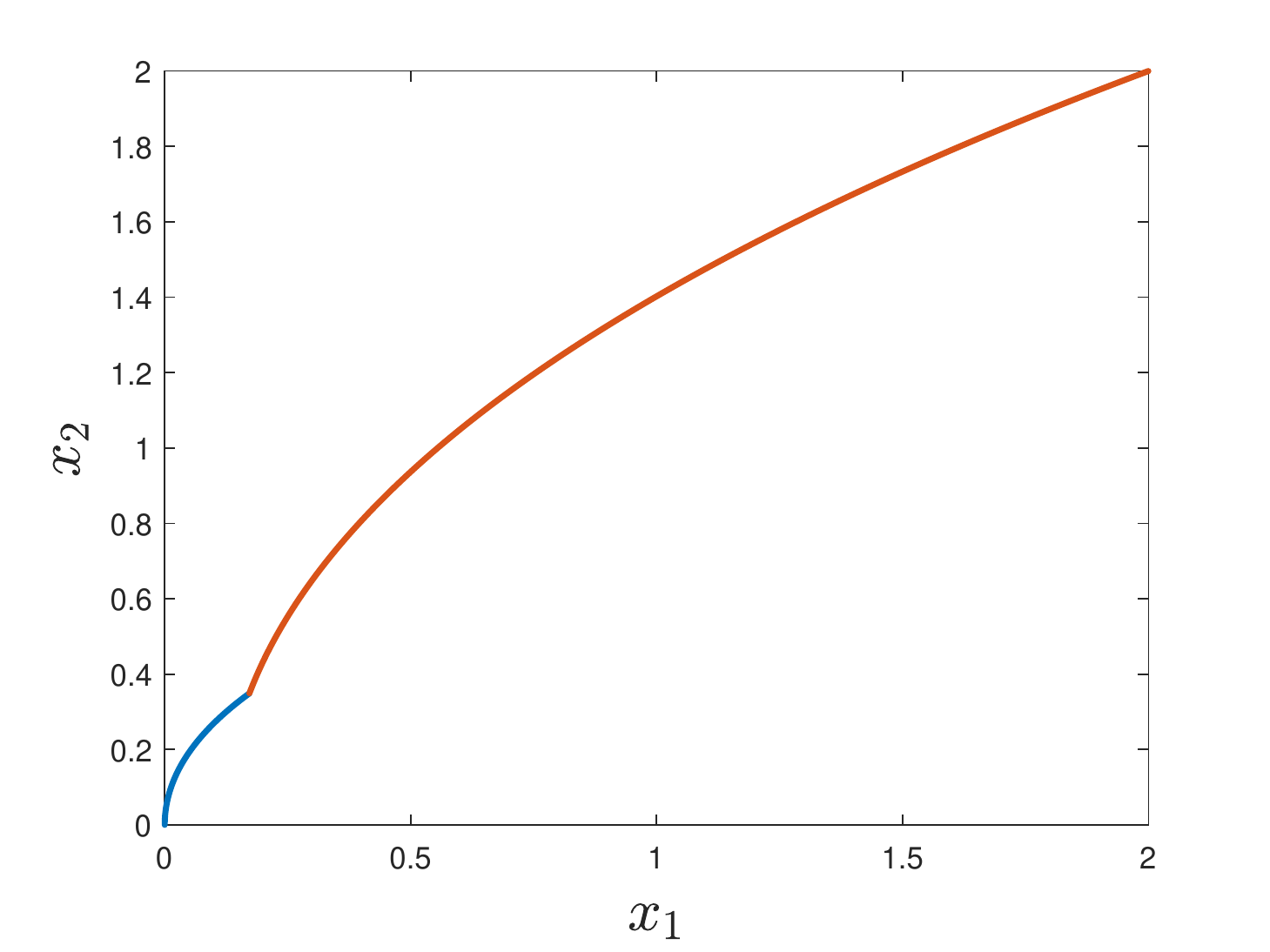}
  \caption{Second-order switching system with time-varying switching interface: $x$.}
   \label{fig:tvswitching}
  \end{center}
  \vspace{-15pt}
\end{figure}
\begin{figure}
  \begin{center}
  \includegraphics[width=3.5in]{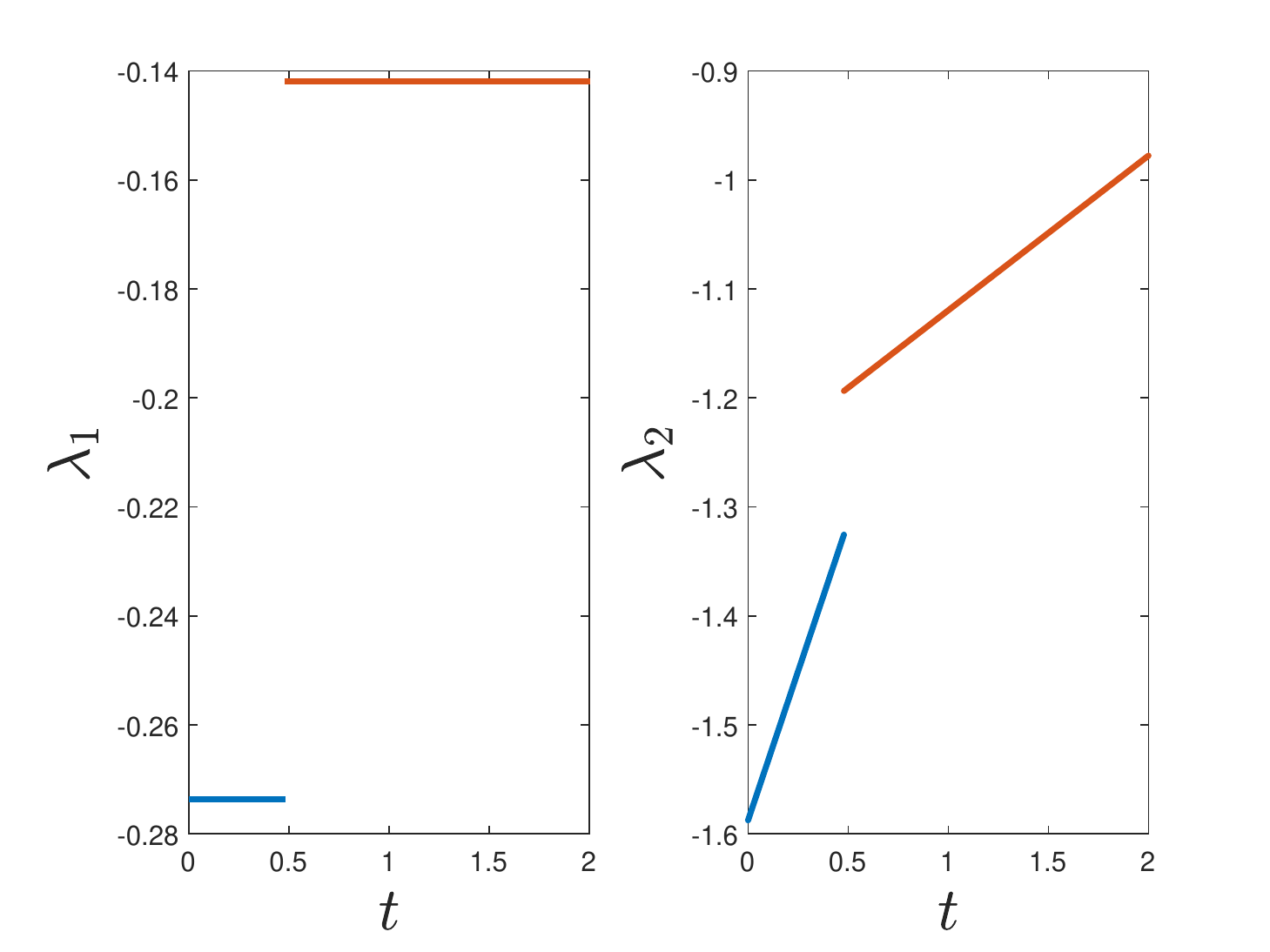}
  \caption{Second-order switching system with time-varying switching interface: $\lambda$.}
   \label{fig:tvswitching_lambda}
  \end{center}
  \vspace{-20pt}
\end{figure}

\section{Application scenario} \label{sec:application}
In this section, we demonstrate the effectiveness of the two jump laws of co-state (i.e., \eqref{eq:finalresult} and \eqref{eqn:jumplawTV}) in solving optimal control problems for state-dependent switching systems.
Combined with the Matlab toolbox ``bvp4c'', we can solve such problems efficiently and precisely.

\subsection{Proposed algorithm (GEL)}
Suppose we have $x\in \mathbb{R}^{n}$. The TIV switching interface is $g(x)=0$ and the switching time is $\tau$.
To avoid confusion, for the following sections, we claim that $f_{-}(x,u,t)$ and $f_{+}(x,u,t)$ represent the system dynamics before switching and after switching respectively.
Similarly, $L_{-}(x,u,t)$ and $L_{+}(x,u,t)$ represent the stage cost before switching and after switching respectively.
The initial state is $x_0$, the final time is $t_f$, the final state is $x_f$.
Then we have the following set of constraints for the TIV switching case: 
\begin{align}
\dot x_{+}(t) &= f_{+}(x,u),  \label{eqn: state_p}\\
\dot x_{-}(t) &= f_{-}(x,u),  \label{eqn:state_n}\\
\dot \lambda_{+}(t) &= -\frac{\partial H_{+}}{\partial x} = h_{+}(x,u)  \label{eqn:costate_p}\\
\dot \lambda_{-}(t) &= -\frac{\partial H_{-}}{\partial x} = h_{-}(x,u)   \label{eqn:costate_n}\\
\frac{\partial H_{+}}{\partial u} &= 0 \rightarrow \phi_{+}(\lambda_{+},u_{+}) = 0 \\
\frac{\partial H_{-}}{\partial u} &= 0  \rightarrow \phi_{-}(\lambda_{-},u_{-}) = 0\\
x_{+}(\tau) &= x_{-}(\tau) \\
g(x_{\pm}(\tau)) &= 0 \label{eqn:swi_left}\\
x_{-}(t_0) &= x_0\\
 x_{+}(t_f) &= x_f  \\
\lambda_{+}(\tau) - \lambda_{-}(\tau)& = -\frac{(L_{+}(\tau)-L_{-}(\tau)+\langle\lambda(\tau)\rangle^\top \Delta f(\tau))m(\tau)}{\langle m(\tau)^\top f(\tau)\rangle}.\label{eq:equivalenceP}
\end{align}
where $m$ is the normal vector of the switching interface at the switching point.
In the above problem, $x_{+}(t)$, $x_{-}(t)$, $\lambda_{+}(t)$, $\lambda_{-}(t)$, $u_{+}(t)$, $u_{-}(t)$, $\tau$ are unknown.
So there are 6n+1 unknowns and 6n+1 constraints.
Eliminating $u$, if we know this jump law equation, we can solve the original problem by solving $4n$ differential equations (related to $x$ and $\lambda$, $x \in \mathbb{R}^{n}$, $\lambda \in \mathbb{R}^{n}$) and one unknown $\tau$ with $4n+1$ boundary constraints.
When the systems have time-varying switching interface, we only need to replace the jump law in~\eqref{eq:equivalenceP} with \eqref{eqn:jumplawTV}.
\subsection{Examples}
We propose an iterative algorithm (called GEL) to solve the optimal control of such systems efficiently leveraging the (co-state) jump laws in \eqref{eq:finalresult} and \eqref{eqn:jumplawTV}.
The main idea is to construct the following function:
\begin{align} \label{eqn:NewF}
F(\tau) &= \Delta \lambda(\tau) -G(\lambda_{+},\lambda_{-},\tau, x_{-},x_{+})\\
&=
\lambda_{+}-\lambda_{-}-G(\lambda_{+},\lambda_{-},\tau, x_{-},x_{+}),
\end{align}
where 
\begin{align*}
 &G(\lambda_{+},\lambda_{-},\tau, x_{-},x_{+})= \\ &-\frac{(L_{+}(\tau)-L_{-}(\tau)+\langle\lambda(\tau)\rangle^\top \Delta f(\tau))m(\tau)}{\langle m(\tau)^\top f(\tau)\rangle}.   
\end{align*}
Note that this option is not unique.
We can also choose other boundary constraints such as \eqref{eqn:swi_left} to be the function $F(\tau)$.
By using MATLAB ``bvp4c'', we can solve the optimal control problem of this switching system as two boundary value problems with multiple boundary conditions.
Then we find $\tau$ by Newton's Method (a root-founding process), which is
\begin{align}
\tau_{k+1} = \tau_{k} -\alpha \frac{F(\tau)}{F'(\tau)},
\end{align}
where $\alpha$ is the step size.
The stopping criterion is then $|\tau_{k+1}-\tau_{k}|< \mathrm{tol}$ and $|F(\tau_k) - 0| <\mathrm{tol}$.
For this derivative part $\frac{dF}{d\tau} = \frac{F(\tau+\delta \tau)-F(\tau)}{\delta \tau}$, a perturbation scheme is used to obtain $F(\tau+\delta \tau)$.
We summarize the proposed approach in \textbf{Algorithm~\ref{alg:cap}}.
\begin{algorithm}
\caption{Proposed algorithm (named GEL) for solving hybrid system optimal control problem numerically}\label{alg:cap}
\begin{algorithmic}
\State Initialize $\tau_0 = 0.5$, tolerance $\mathrm{tol}=0.0001$, $\delta \tau = 0.1$\;
\For {$k=0:iter$}
\State Solve two-point boundary value problem \eqref{eqn:state_n}, \eqref{eqn:costate_n}, \eqref{eqn: state_p}, \eqref{eqn:costate_p} using bvp4c\;
\State Calculate $F(\tau_k)$ \;
\If{$ |F(\tau_k)-0| < tol$}
\State break
\Else
\State Add a small perturbation $\tau'_k = \tau_k+\delta \tau$\;
\State Calculate $F(\tau'_k)$
\State Calculate $F'(\tau_k) = \frac{F(\tau'_k)-F(\tau_k)}{\delta \tau}$
\State Newton Update: $\tau_{k+1} = \tau_{k} - \frac{F(\tau_k)}{F'(\tau_k)}$\;
\If {$|\tau_{k+1}-\tau_{k}|< tol$}
\State break
\EndIf
\EndIf
\EndFor
\end{algorithmic}
\end{algorithm}

As the authors know, there is no direct way to solve this problem.
Thus we use the ICLOCS2 \cite{nie2018iclocs2} to brute force search the switching time $\tau$ and switching states, and regard the hybrid system into two systems.
In this way, each sub-problem is an optimal control for a continuous system with some known boundary conditions.
This indirect use of ICLOCS2 toolbox is very time-consuming.
Think of a second-order system with a switching interface $x_1+x_2 = 0$.
We need to use two for-loops to search the switching time $\tau$ and switching state $x_1(\tau)$.
In \cite{Shaikh}, the authors proposed an efficient gradient-descent-based algorithm named HMPMAS for solving hybrid switching systems, which is proven much more efficient than the algorithm in \cite{Xuping}.
We compare our algorithm with the algorithm in \cite{Shaikh} on three different systems (one time-varying system with a time-varying switching interface, two are from the paper \cite{Shaikh}).
Given a problem, we compare the time required by different algorithms to generate a solution with the same level of numerical precision.
All the computations are performed using MATLAB 2020b on a personal computer with i7-8665U CPU and 16GB RAM.
Each experiment is run 5 times and the average CPU time is reported for responsible comparison.
\subsubsection{Example 1}
In this example, we use our algorithm to solve a one-order system with a time-varying switching interface.
$x_0 = 0$, $x_f = 2$, $g(x,t)=x+t-1=0$, $t_f=2$.
The system dynamics is
\begin{align*}
\left\{\begin{matrix}
\dot x = tx+ u,\quad x+t\leq 1\\ 
\dot x = u, \quad x+t \geq 1
\end{matrix}\right.,
\end{align*}
while the stage cost
\begin{align}
\left\{\begin{matrix}
L = \frac{1}{2}(x^2+u^2),\quad  x+t \leq  1\\ 
L = \frac{1}{8}(x^2+u^2),\quad  x+t \geq 1
\end{matrix}\right.    .
\end{align}
In this example, the jump law of \eqref{eqn:jumplawTV} is used.
The stopping condition is $\mathrm{tol}=0.0001$. Initial value is $\tau_0 = 0.5$, $x_s^0=0.5$ (initialization of state, which is required by the HMPMAS algorithm).
For the HMPMAS algorithm, the step size is set as $r_k = 0.4$ compromising the convergence speed and converge precision.
Table \ref{tbl:Example2} shows the results compared with \cite{Shaikh}.
We do not present the cost for HMPMAS because the switching state $x_s$ does not reach the switching manifold.
This pair of $(\tau,x_s)$ is not on the switching interface even though this algorithm indeed converges.
For the ``-'' in the ICLOCS2 row, this means the time is too long to obtain the optimal solution and not comparable to the other two algorithms.
\begin{table}[!htp]
\centering
\vspace{-5pt}
\caption{Algorithm Performance Comparison: Example 1}
\vspace{-3pt}
\begin{tabular}{|c||c||c||c||c|}
\hline
Algorithm  & $\tau$ & $J$ & Time (s) &Iteration\\
\hline
ICLOCS2 & 0.7495 & 0.5558 & - & -\\ \hline
 HMPMAS &  \shortstack{0.7230 \\ ($x_s=0.2368$)} & n/a &
 31.0038
 & 19\\ \hline
GEL (Proposed)  &  0.7495   & 0.5558  &16.8556 & 5
 \\\hline
\end{tabular}
\label{tbl:Example2}
\end{table}

\figref{fig:GEL_algo_eg2_x} shows the optimal trajectory of $x$.
In this figure, ``SM'' refers to the switching interface.
It shows that the proposed algorithm has a more precise solution.
\figref{fig:GEL_algo_eg2_tau} plots the trends of absolute error of switching time $\tau$ versus the number of iterations. 
The true optimal $\tau$ is obtained from the toolbox ICLOCS2. 
Formally, the absolute error is given by
\begin{equation} \label{eqn:abserror}
\mathrm{Error}= |\tau-\tau_{\text{ICLOCS2}}|.
\end{equation}
\figref{fig:GEL_algo_eg2_tau} shows that the proposed algorithm needs fewer iterations to reach the global optimal.
\begin{figure} [!htp]
  \begin{center}
  \includegraphics[width=3.5in]{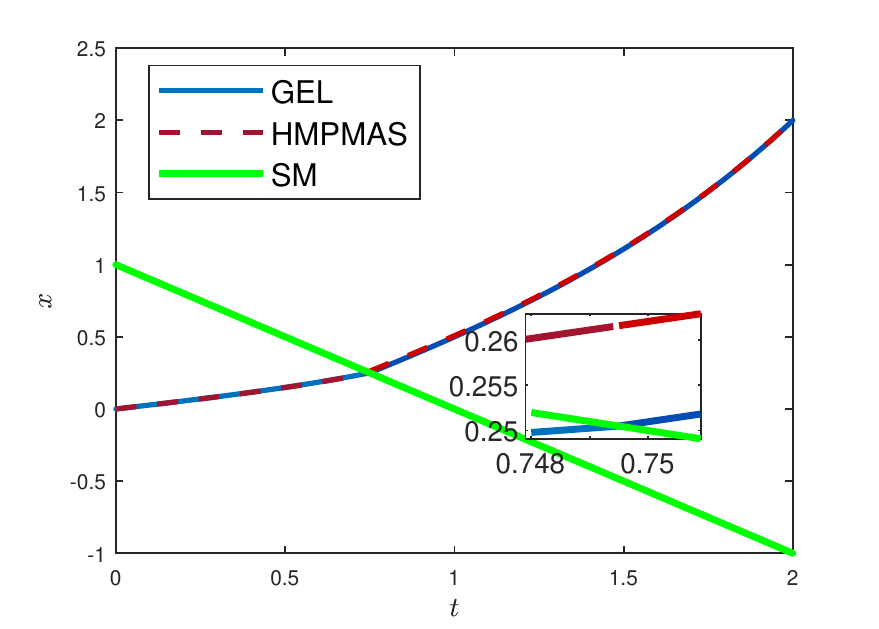}
  \caption{A time-varying system with time-varying switching interface: $x$.}
  \label{fig:GEL_algo_eg2_x}
  \end{center}
  \vspace{-15pt}
\end{figure}
\begin{figure} [!htp]
  \begin{center}
  \includegraphics[width=3.5in]{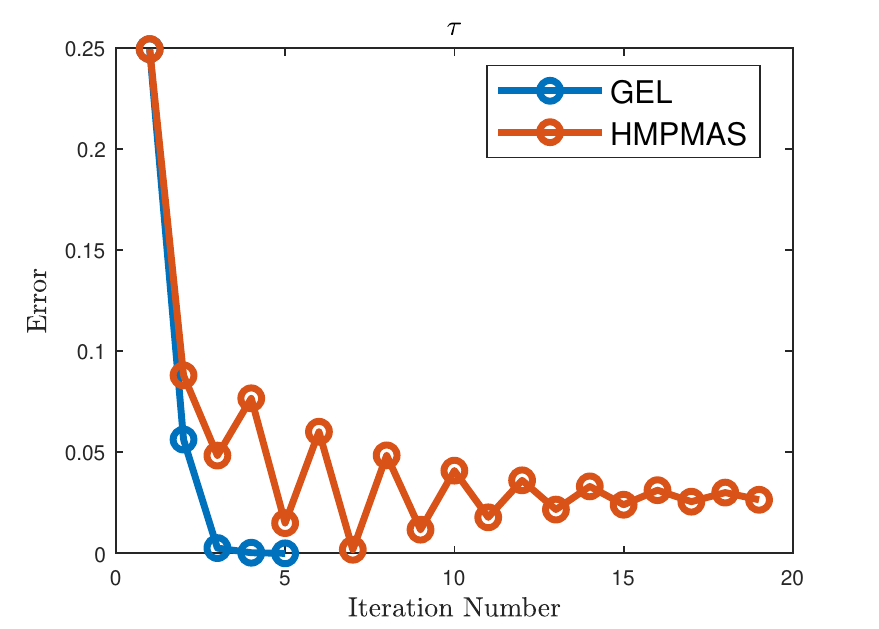}
  \caption{A time-varying system with time-varying switching interface: absolute error of $\tau$.}
  \label{fig:GEL_algo_eg2_tau}
  \end{center}
  \vspace{-15pt}
\end{figure}
\subsubsection{Example 2}
We then consider a second-order system:
\begin{equation*}
\mathrm{s1}: \quad \dot x = \begin{bmatrix}
1.5 & 0\\ 
0 & 1
\end{bmatrix}x+\begin{bmatrix}
1\\ 
1
\end{bmatrix}u,
\end{equation*}
switches to
\begin{equation*}
\mathrm{s2}: \quad \dot x = \begin{bmatrix}
0.5 & 0.866\\ 
0.866 & -0.5
\end{bmatrix}x+\begin{bmatrix}
1\\ 
1
\end{bmatrix}u,    
\end{equation*}
with switching interface $x_1+x_2-7=0$.
The cost function to be minimized is
\begin{equation*}
J(u) = \frac{1}{2}(x_1(t_f)-10)^2 +\frac{1}{2}(x_2(t_f)-6)^2+\frac{1}{2}\int_{0}^{t_f}u^2(t) {\rm d}t   .
\end{equation*}
The initial time, final time, and initial state are $t_0=0$, $t_f=2$ and $x_0=[1,1]^\top$ respectively.
The initial value is set as $\tau_0 = 1.5$, $x_s^0 = [4.5,2.5]^\top$ (same as \cite{Shaikh}), $\mathrm{tol}=0.0001$, $r_k = 0.02\frac{k+1}{k+2}$, $\alpha = 1$.
Table \ref{tbl:Example3} shows the result.
It shows that the proposed algorithm needs less iteration and time to find an optimal value.
Moreover, the proposed algorithm has higher precision compared to that in \cite{Shaikh}.

\begin{table}[!htp]
\caption{Algorithm Performance Comparison: Example 2}
\vspace{-3pt}
\centering
\begin{tabular}{|c||c||c||c||c|} \hline
Methods &   $\tau$ $(x_s)$ & $J$ & Time (s) & Iteration\\ \hline
ICLOCS2 &  \shortstack{1.1624\\ (4.5556, 2.4444)} & 0.1130 &   - &  -\\ \hline
HMPMAS  & \shortstack{1.1630 \\(4.5456, 2.4326)} & n/a & 83.1040& 20\\ \hline
Proposed &  \shortstack{ 1.1625 \\ (4.5562, 2.4438)}  &  0.1130  & 21.3326 &  5\\ \hline
\end{tabular}\\ 
\label{tbl:Example3}
\vspace{-10pt}
\end{table}

In this example, one finds that despite keeping the initialization of the state $x_s$ on the switching interface while $\tau$ starting from different initial values, the convergence to the switching interface is not guaranteed for the HMPMAS algorithm.
However, for the proposed algorithm, we only need to search for one variable $\tau$ and it converges faster.
\figref{fig:GEL_algo_eg3_x} shows the optimal trajectory of $x$.
\figref{fig:GEL_algo_eg3_tau} is the iterative curve of absolute error of switching time $\tau$ as defined in \eqref{eqn:abserror}.
\begin{figure} [!htp]
  \begin{center}
  \includegraphics[width=3.5in]{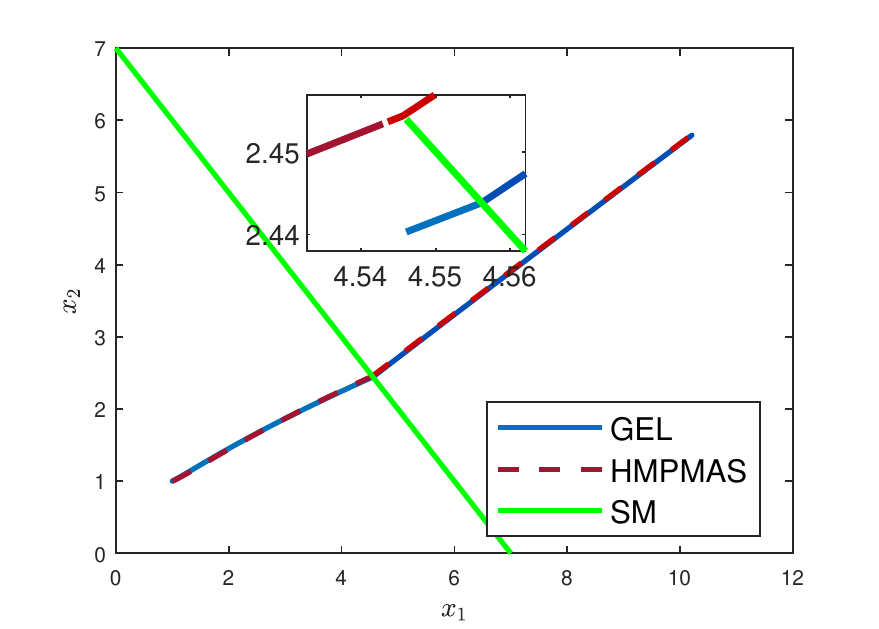}
  \caption{A second order system with time-invariant switching interface: $x$.}
  \label{fig:GEL_algo_eg3_x}
  \end{center}
  \vspace{-10pt}
\end{figure}
\begin{figure} [!htp]
  \begin{center}
  \includegraphics[width=3.5in]{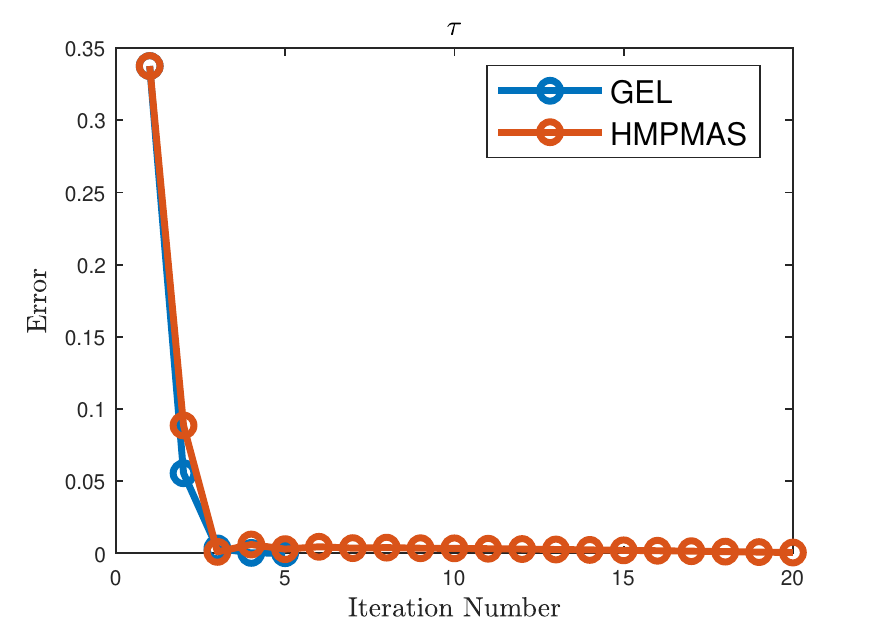}
  \caption{A second order system with time-invariant switching interface: absolute error of $\tau$.}
  \label{fig:GEL_algo_eg3_tau}
  \end{center}
  \vspace{-15pt}
\end{figure}
\subsubsection{Example 3}
The following system:
\begin{align*}
&\mathrm{s1}: \quad \dot x = x+ u x \\
&\mathrm{s2}: \quad \dot x = -x+x u,
\end{align*}
with cost function $J = \int_0^2 \frac{1}{2}u^2 \rm{d} t$.
$t_0=0$, $t_f=2$, $m(x,t)=x-e t=0$, $x_0=1$, $x_f=1$.
The stopping condition is $\mathrm{tol}=0.0001$.
The initial values for HMPMAS algorithm are set as $\tau_0 = 0.5$, $x_s^0 = 0.1$ and step size $r_k =0.04\frac{k+1}{k+2}$.
Table \ref{tbl:Example4} is the result.
Only 5 iterations are needed to reach convergence using the proposed algorithm.
We also find that the HMPMAS algorithm is very fragile to the step size $r_k$.
One has to play with the step size to find the best one to make this algorithm converge.
\figref{fig:GEL_algo_eg4_x} is the trajectory of $x$ using proposed algorithm and HMPMAS.
\figref{fig:GEL_algo_eg4_tau} is the absolute error curve of $\tau$ using proposed algorithm and HMPMAS.
\begin{table} [!htp]
\caption{Algorithm Performance Comparison: Example 3}
\vspace{-7pt}
\centering
\begin{tabular}{|c||c||c||c||c|} \hline
Methods &   $\tau$ & $J$ & Time (s) & Iteration\\ \hline
ICLOCS   &  1.0000   &   0 & -   &- \\ \hline
HMPMAS  &  0.9678  &  0.0032 & 576.2577 & 102   \\ \hline
Proposed & 1.0004  &  8.1008e-04  & 45.2587& 5
 \\ \hline
\end{tabular}
\label{tbl:Example4}
\vspace{-10pt}
\end{table}

\begin{figure} 
  \begin{center}
  \includegraphics[width=3.5in]{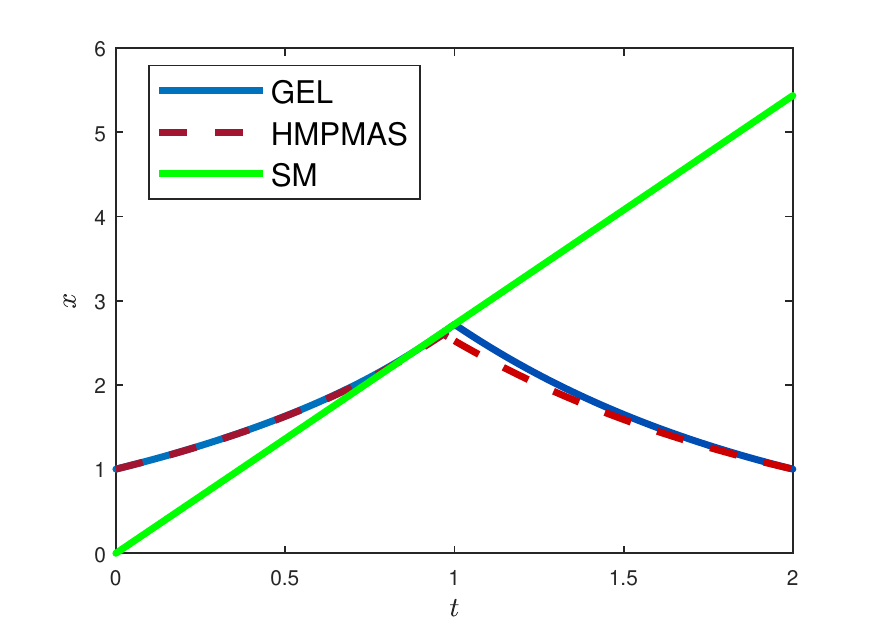}
  \caption{A singular system with time-varying switching interface: $x$.}
  \label{fig:GEL_algo_eg4_x}
  \end{center}
  \vspace{-20pt}
\end{figure}
\begin{figure} 
  \begin{center}
  \includegraphics[width=3.5in]{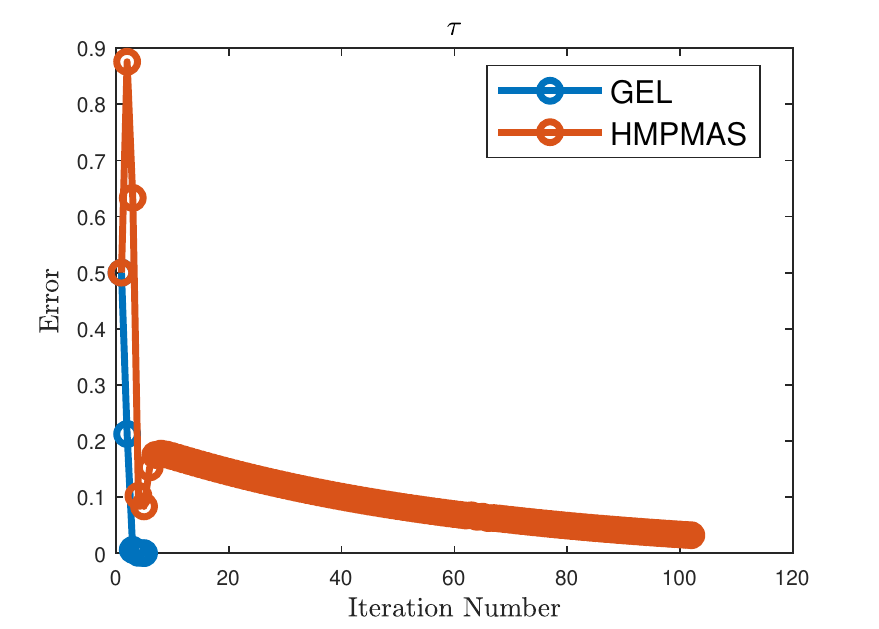}
  \caption{A singular system with time-varying switching interface: absolute error of $\tau$.}
  \label{fig:GEL_algo_eg4_tau}
  \end{center}
  \vspace{-20pt}
\end{figure}
\section{Conclusion} \label{sec:conclusion}
In this paper, we derived the jump law of co-state in optimal control for state-dependent switching systems, 
based on which we developed an efficient algorithm for solving optimal control problems with hybrid systems.
Future work will be extended to the following aspects:
\begin{itemize}
    \item Investigate the generalized function aspects to derive the jump law of co-state directly from the Euler-Lagrange equation along with the ideas of \cite{Hyun2016CausalIM2} and \cite{GELACC_Mi}.
    \item Study the co-state in optimal control of state-dependent switched systems with time delay.
\end{itemize}

\bibliographystyle{IEEEtran}
\bibliography{IEEEabrv,Bibliography}

\begin{thebibliography}{10}
\providecommand{\url}[1]{#1}
\csname url@rmstyle\endcsname
\providecommand{\newblock}{\relax}
\providecommand{\bibinfo}[2]{#2}
\providecommand\BIBentrySTDinterwordspacing{\spaceskip=0pt\relax}
\providecommand\BIBentryALTinterwordstretchfactor{4}
\providecommand\BIBentryALTinterwordspacing{\spaceskip=\fontdimen2\font plus
\BIBentryALTinterwordstretchfactor\fontdimen3\font minus
  \fontdimen4\font\relax}
\providecommand\BIBforeignlanguage[2]{{%
\expandafter\ifx\csname l@#1\endcsname\relax
\typeout{** WARNING: IEEEtran.bst: No hyphenation pattern has been}%
\typeout{** loaded for the language `#1'. Using the pattern for}%
\typeout{** the default language instead.}%
\else
\language=\csname l@#1\endcsname
\fi
#2}}

\bibitem{Liu2014OptimalCO}
C.~Liu and Z.~Gong, ``Optimal control of switched systems arising in
  fermentation processes,'' 2014.

\bibitem{multiphase}
\BIBentryALTinterwordspacing
J.~T. Betts, \emph{Practical Methods for Optimal Control and Estimation Using
  Nonlinear Programming, Second Edition}, 2nd~ed.\hskip 1em plus 0.5em minus
  0.4em\relax Society for Industrial and Applied Mathematics, 2010. [Online].
  Available: \url{https://epubs.siam.org/doi/abs/10.1137/1.9780898718577}
\BIBentrySTDinterwordspacing

\bibitem{MITcheetah}
H.-W. Park, P.~Wensing, and S.~Kim, ``High-speed bounding with the mit cheetah
  2: Control design and experiments,'' \emph{The International Journal of
  Robotics Research}, vol.~36, p. 027836491769424, 03 2017.

\bibitem{Magnusrobots}
M.~Egerstedt, ``Behavior based robotics using hybrid automata,'' in
  \emph{Hybrid Systems: Computation and Control}, N.~Lynch and B.~H. Krogh,
  Eds.\hskip 1em plus 0.5em minus 0.4em\relax Berlin, Heidelberg: Springer
  Berlin Heidelberg, 2000, pp. 103--116.

\bibitem{skating}
T.~Mehta, D.~Yeung, E.~Verriest, and M.~Egerstedt, ``Optimal control of
  multi-dimensional, hybrid ice-skater model,'' 08 2007, pp. 2787 -- 2792.

\bibitem{snelllaw}
M.~S. Shaikh and P.~E. Caines, ``On relationships between weierstrass-erdmannn
  corner condition, snell's law and the hybrid minimum principle,'' in
  \emph{2007 International Bhurban Conference on Applied Sciences Technology},
  2007, pp. 117--122.

\bibitem{Liberzon2013SwitchedS}
D.~Liberzon, ``Switched systems : Stability analysis and control synthesis,''
  2013.

\bibitem{Daafouz}
J.~Daafouz, P.~Riedinger, and C.~Iung, ``Stability analysis and control
  synthesis for switched systems: a switched lyapunov function approach,''
  \emph{IEEE Transactions on Automatic Control}, vol.~47, no.~11, pp.
  1883--1887, 2002.

\bibitem{FengSurvey}
F.~Zhu and P.~Antsaklis, ``Optimal control of hybrid switched systems: A brief
  survey,'' \emph{Discrete Event Dynamic Systems}, vol.~25, 09 2014.

\bibitem{Stellato}
B.~Stellato, S.~Ober-Blöbaum, and P.~J. Goulart, ``Optimal control of
  switching times in switched linear systems,'' in \emph{2016 IEEE 55th
  Conference on Decision and Control (CDC)}, 2016, pp. 7228--7233.

\bibitem{Verriest_M3D}
E.~I. Verriest, ``Multi-dimensional multi-mode systems: Structure and optimal
  control,'' in \emph{49th IEEE Conference on Decision and Control (CDC)},
  2010, pp. 7021--7026.

\bibitem{PiecewiseAffine}
F.~Christophersen, ``Optimal control of constrained piecewise affine systems,''
  \emph{Lecture Notes in Control and Information Sciences}, vol. 359, 01 2007.

\bibitem{Shaikh}
M.~S. Shaikh and P.~E. Caines, ``On the hybrid optimal control problem: Theory
  and algorithms,'' \emph{IEEE Transactions on Automatic Control}, vol.~52,
  no.~9, pp. 1587--1603, 2007.

\bibitem{Xuping}
X.~Xu and P.~Antsaklis, ``Optimal control of switched systems based on
  parameterization of the switching instants,'' \emph{IEEE Transactions on
  Automatic Control}, vol.~49, no.~1, pp. 2--16, 2004.

\bibitem{Xuping2002}
------, ``Optimal control of switched systems via nonlinear optimization based
  on direct differentiation of value functions,'' \emph{International Journal
  of Control}, vol.~75, pp. 1406--1426, 11 2002.

\bibitem{Witsenhausen}
H.~Witsenhausen, ``A class of hybrid-state continuous-time dynamic systems,''
  \emph{IEEE Transactions on Automatic Control}, vol.~11, no.~2, pp. 161--167,
  1966.

\bibitem{optimalcontrolbook}
L.~Arturo, \emph{Optimal Control: An Introduction}.\hskip 1em plus 0.5em minus
  0.4em\relax Birkhäuser, 2001.

\bibitem{nie2018iclocs2}
Y.~Nie, O.~Faqir, and E.~C. Kerrigan, ``Iclocs2: Solve your optimal control
  problems with less pain,'' 2018.

\bibitem{Hyun2016CausalIM2}
N.~s.~P.~Hyun and E.~I. Verriest, ``A causal interpretation of nonlinear
  impulsive system based on non standard analysis,'' \emph{Nonlinear Analysis:
  Hybrid Systems}, p. 138–154, 2017.

\bibitem{GELACC_Mi}
M.~Zhou and E.~I. Verriest, ``Generalized euler-lagrange equation: A challenge
  to schwartz’s distribution theory*,'' in \emph{2022 American Control
  Conference (ACC)}, 2022, pp. 4951--4956.

\end{thebibliography}

\vfill
\end{document}